\documentclass[12pt]{amsart}
\usepackage{color}

\usepackage{amssymb}
\usepackage{graphicx}

\newcommand{\CC}{{\mathbb C}}

\newcommand{\RR}{{\mathbb R}}

\newcommand{\defeq}{\stackrel{\rm{def}}{=}}

\newcommand{\rest}{|}

\renewcommand{\Re}{\mathop{\rm Re}\nolimits}
\renewcommand{\Im}{\mathop{\rm Im}\nolimits}

\newcommand{\wV}{{\widetilde V}}
\newcommand{\wM}{{\widetilde M}}
\newcommand{\wo}{{\widetilde \omega}}

\theoremstyle{plain}

\newtheorem{prop}{Proposition}

\theoremstyle{definition}

\numberwithin{equation}{section}

\def\squarebox#1{\hbox to #1{\hfill\vbox to #1{\vfill}}} 
 
\newcommand{\sech}{\textnormal{sech}}

\usepackage{amsxtra}

\ifx\pdfoutput\undefined
  \DeclareGraphicsExtensions{.pstex, .eps}
\else
  \ifx\pdfoutput\relax
    \DeclareGraphicsExtensions{.pstex, .eps}
  \else
    \ifnum\pdfoutput>0
      \DeclareGraphicsExtensions{.pdf}
    \else
      \DeclareGraphicsExtensions{.pstex, .eps}
    \fi
  \fi
\fi

\title
[Geometric structure of NLS evolution]
{Geometric structure of NLS evolution}

\author[J. Holmer]
{Justin Holmer}
\email{holmer@math.brown.edu}
\address{Department of Mathematics, Brown University\\
151 Thayer Street, Providence, RI 02912, USA}
\author[M. Zworski]
{Maciej Zworski}
\email{zworski@math.berkeley.edu}
\address{Mathematics Department, University of California \\
Evans Hall, Berkeley, CA 94720, USA}

\begin{document}    
   
\maketitle

\section{Introduction}
\label{int}

The purpose of this note is to clarify the relation
between the Hamiltonian and Lagrangian approaches to 
nonlinear evolution equations. In particular we explain 
the least action principle and the Noether theorem in this 
context.
The specific 
point of view,
adapted from Souriau's book \cite{Sou}, has perhaps not been 
presented for nonlinear evolution equations. It seems that
in the mathematics literature, as in \cite{FrSi},\cite{GSS},\cite{HZ1}, and in 
references
given there, the Hamiltonian point of view is prevalent.
In the physics  literature, 
see for instance \cite{GHW},\cite{OS}, the Lagrangian point of view rules
with the symplectic structure largely neglected. In \S \ref{ed} we 
present one possible mathematical reason for that.

\section{The Hamiltonian structure}
\label{rhs}

In this section we recall well known facts about the
Hamiltonian structure of the nonlinear Schr\"odinger equation.
The same point of view applies to other evolution equations, see
for instance \cite{GSS} and references given there. 

For simplicity we will consider the case of dimension one, and 
\[ V \defeq H^1(\mathbb{R}, \mathbb{C})\subset L^2(\mathbb{R},\mathbb{C})\,,\]
viewed as a {\em real} Hilbert space.
The inner product and the symplectic form are given by 
\begin{equation}
\label{eq:omega}\
\langle u,v \rangle  \defeq \Re \int u\bar v \,, \  \ 
\omega(u,v) \defeq \langle  u , i v \rangle = \Im \int u\bar v\,, 
\end{equation}
Let 
$H:V\to \mathbb{R}$ be a function, a Hamiltonian.  
The associated Hamiltonian
vector field is a map $\Xi_H : V\to TV$, 
which means that for a particular point $u\in V$, we have 
$(\Xi_H)_u \in T_uV$. The vector field  $\Xi_H$ is defined by the relation
\begin{equation}
\label{eq:Hamvf} \omega(v , (\Xi_H)_u) = d_uH(v)
\,, \end{equation}
where $v\in T_uV$, and $d_uH:T_uV \to \mathbb{R}$ is defined by
$$d_uH(v) = \frac{d}{ds}\Big|_{s=0} H(u+sv) \,. $$
In the notation above
\begin{equation}
\label{eq:nats}  dH_u ( v ) = \langle dH_u , v \rangle \,, 
\ \  (\Xi_H)_u = \frac 1 i dH_u  \,. 
\end{equation}

If we take $ V = H^1 ( \RR , \CC ) $ with the symplectic form \eqref{eq:omega},
and
$$H(u) = \int \frac14|\partial_x u|^2 - \frac1 {p+1} |u|^{p+1} $$
then we can compute
\begin{align*}
d_uH(v) &= \Re \int ( (1/2) \partial_x u \partial_x \bar v - |u|^{p-1} u \bar v)\\
&= \Re \int ( - (1/2)\partial_x^2 u - |u|^{p-1} u)\bar v \,. 
\end{align*}
Thus, in view of \eqref{eq:nats} and \eqref{eq:Hamvf}, 
$$(\Xi_H)_u = \frac 1 i \left( - \frac12 \partial_x^2u - |u|^{p-1}u \right) $$
The flow associated to this vector field (Hamiltonian flow) is
\begin{equation}
\label{eq:Hflow}
\dot u = (\Xi_H)_u =  \frac 1 i  \left( - 
\frac12 \partial_x^2u - |u|^{p-1}u \right) 
\,.
\end{equation}

\section{The Lagrangian point of view and the Noether theorem}
\label{lpv}

According to \cite{Sou} the following point of view towards
dynamics goes back to Lagrange. We consider 
\[  \wV = V \times \RR = H^1 ( \RR , \CC ) \times \RR \,,\]
and the following one form on $ \wV $ (we are rather informal
here about dual spaces etc):
\begin{equation}
\label{eq:defal}
 \alpha_{(u,t)} ( v , T ) \defeq \frac 12 \omega ( u , v ) -  H ( u ) T \,, \ \ 
( v , T ) \in T_{ ( u , t ) } \wV \,, \ \ ( u , t ) \in \wV \,. 
\end{equation}

\medskip

\noindent
{\bf Remark.} The presence of the factor $ 1/2 $ in front of $ \omega $
in the definition of $ \alpha $ is best understood using the
finite dimensional analogy: if $ z = x + i \xi $, $ x, \xi \in \RR^n $, then
\begin{equation}
\label{eq:alpha}   \alpha = \frac{1}2 \Im z d\bar z   - H( x, \xi ) dt  
= \frac 12 ( \xi dx - x d \xi ) - H ( x, \xi ) dt \,.
\end{equation}

\medskip

We then define the differential of $ \alpha $:
\[ \wo \defeq d \alpha \,, \]
that is
\[ \wo_{(u,t)} ( ( v_1 , T_1 ) , (v_2 , T_2 ) ) = \omega ( v_1 , v_2 ) - 
dH_{u} ( v_1 ) T_2 + dH_u ( v_2 ) T_2 \,, \]
where we used the notation of \S \ref{rhs}. This calculation is
easily understood using the analogy with \eqref{eq:alpha}:
\[ d ( \xi dx - x d \xi )/2 = d \xi \wedge d x \,. \]

Having $ \wo $ makes
$ \wV $ a presymplectic space in the sense that $ \wo $ has
a kernel of dimension one. Here, the kernel is
\[ \ker \wo_{(u,t)} \defeq \left\{ ( v, T )  \in T_{ ( u , t) } \wV \; ; \; \
\forall \; (v', T') \in T_{ ( u , t ) } \wV \,, \
\wo_{(u,t)} ( ( v, T ) , ( v' , T') ) = 0 \right\} \,. \]

The following proposition replaces \eqref{eq:Hflow} with a condition
related to $ \wo = d \alpha $:

\begin{prop}
\label{p:HL}
The curve $ t \mapsto u ( t ) \in V $ is a solution to 
\begin{equation}
\label{eq:NLS}
  i u_t = - \frac 12 \partial_x^2 u - |u|^{p-1} u \,, 
\end{equation}
if and only if 
\begin{equation}
\label{eq:HL}
(\dot u ( t ) , 1 ) \in \ker \wo_{ u ( t ) } \,. 
\end{equation}
In other words,
\[ \ker \wo_{u} = \RR ( \Xi_{ u } , 1 ) \,. \]
\end{prop}
\begin{proof}
We already know that \eqref{eq:NLS} is equivalent to \eqref{eq:Hflow}.
We then check that 
\[ \begin{split}
\wo ( ( (\Xi_H)_u , 1 ) , ( v , T) ) & = 
\omega ( (\Xi_H)_u , v  ) - \langle dH_u , T (\Xi_H)_u  - v  \rangle
\\
&  = - \langle dH_u , v  ) -  \langle dH_u ,  T (\Xi_H)_u  - v  \rangle\\
& = T   \langle dH_u , (\Xi_H)_u \rangle 
 = T  \langle 
dH_u , (1/i ) dH_u \rangle = 0 \,.
\end{split} \]
\end{proof}

A special case of Noether's Theorem (see \cite[(11.12)]{Sou} for a
more general version using the moment map) is now nicely given using
this point of view:

\begin{prop}
\label{p:Noeth} 
Suppose that 
\[ A ( s) \; : \; 
( s , U ) \mapsto U ( s )  \,,  \ \ s \in \RR \,, \ \ U \in \wV \,, \]
is a
one parameter group acting on $ \wV $ and preserving $ \alpha $:
\begin{equation}
\label{eq:Aal}
  A(s)^* \alpha = \alpha \,, \end{equation}
(here the pullback is given by $ f^* \alpha_{(u,t)}  ( v, T ) 
\defeq \alpha_{ f ( u , t ) } 
(f_* ( v , T ) ) $).
Then 
\begin{equation}
\label{eq:Noeth} F ( u, t ) \defeq \alpha_{(u,t)} \left ( \frac{d}{ds} A ( s ) (u,t) |_{s=0} \right) \,, \ 
(u,t)  \in \wV \end{equation}
is conserved by the flow \eqref{eq:NLS}.
\end{prop}
\begin{proof}
In the finite dimensional case we use Cartan's formula: if $ (d/ds) f_s|_{s=0} 
= X  $ (here $ f_s : \wV \rightarrow \wV $), then at $ s = 0$,
\[ \frac{d}{ds} f_s^* \alpha = d ( \alpha ( X ) ) + ( d \alpha ) ( X , \bullet) 
\,. \]
If we take $ f_s = A( s) $ then the left hand side is $ 0 $ and 
$ X = (d/ds) A ( s ) ( u , t ) |_{s=0} $. The invariance of $ F $ is
then equivalent to 
\[  d ( \alpha ( X ) ) ( \dot u , 1 ) = 0 \,\]
but since 
\[  d ( \alpha ( X ) )  = - \wo ( X , \bullet)  \,,\]
this follows from Proposition \ref{p:HL}. The same argument applies 
formally in the case of evolution equations and can be easily verified.
\end{proof}

\subsection{Standard group actions.}
\label{sga} 
The basic group action to consider are
\[ ( u , t ) \mapsto ( e^{ -is } u , t )  \,, \ \ 
( u , t ) \mapsto ( u ( \bullet - s ) , t ) \,, \ \ 
( u , t ) \mapsto ( u , t - s ) \,, \]
and in each case we quickly see that $ A ( s ) ^* \alpha = \alpha $.
In the three cases we have 
\begin{gather*}  ( d/ds ) A( s) ( u , t) |_{s=0} = (- i u , 0 ) \,, \ \
 ( d/ds ) A( s) ( u , t) |_{s=0} = ( -u_x  , 0 ) \,, \\
 ( d/ds ) A( s) ( u , t) |_{s=0} = ( 0  , -1 ) \,, \ \
\end{gather*}
respectively, and the conserved quantities obtained using 
the formula \eqref{eq:Noeth} are easily seen to be 
\[ \int |u|^2 dx \,, \ \ \Im \int u_x \bar u dx \,, \ \  H ( u ) \,. \]

A more interesting example is given by considering the Galilean 
invariance:
\[  A( s ) ( u , t ) = ( A_0 ( s , t ) u , t ) \, , \ \ 
A_0 ( s, t ) u \defeq e^{- i t s^2/2 + i \bullet s } u ( \bullet - s t ) \,. \]
We first check that \eqref{eq:Aal} holds. In fact, 
\[  [A ( s )_*]_{(u,t)} ( v, T ) = ( A_0 ( s, t ) v + \partial_t ( 
A_0 ( s , t) u ) T , T ) = ( A_0 ( s, t ) ( v - (i s^2 u /2 + 
s u_x ) T), T ) \,, \]
and hence
\[ \begin{split} 
  (A( s )^* \alpha)_{( u , t )} ( v ,  T ) & = 
\alpha_{ A ( s ) ( u ,t ) } ( [A ( s )_*]_{(u,t)} ( v, T ) ) \\
& = \omega ( A_0 ( s , t ) u , 
A_0 ( s, t ) ( v - ( i s^2u/2 + 
s u_x ) T ) ) - 2 H ( A_0 ( s , t ) u ) T \\
& = 
\alpha_{( u , t ) } ( v , T )\,, 
\end{split} \]
since
\[ \omega  ( A_0 ( s , t ) u , 
A_0 ( s, t )  v) = \omega ( u , v ) \,, \]
and
\[ \begin{split}
H ( A_0 ( s, t ) u ) & = H ( u ) + s^2 \langle u , u \rangle/4 
+ s \langle i u , u_x \rangle/2  \\
& = H ( u ) + s^2 \omega ( i u , u )/4 
+ s \omega ( u , u_x )/2 \,. 
\end{split}
\]

We also see that
\[ \frac{d}{ds} A ( s ) ( u , t ) |_{ s= 0 } = 
( i x u - t u_x , 0 ) \,, \]
formula \eqref{eq:Noeth} gives
\[  F( u , t ) = t \Im \int u_x \bar u dx - \int x |u|^2 dx = 
F ( u , 0 ) = - \int x |u|^2 dx = 0 \,, \]
which of course corresponds to $ p = m q / t $ where
\[ p = \Im \int u_x \bar u dx \,, \ \ q = \frac1m \int x |u|^2 dx \,, \ \
m = \int |u|^2 dx \,, \]
are the momentum, position, and mass, respectively.

\subsection{Scaling}
\label{sc}

Let us now consider another group action preserving 
solutions of \eqref{eq:NLS} ($p > 1$):
\begin{equation}
\label{eq:sc}  
( s , u , t ) \longmapsto ( A_0 ( s ) u , 
s^{-2} t ) \,, \ \ 
A_0 ( s ) u ( \bullet ) \defeq s^{\frac{2}{p-1}} u (  s \bullet )  \,. 
\end{equation}
Then 
\[ [A ( s )_*]_{(u,t)} ( v , T ) = A(s) ( v, T ) \,, \]
and 
\[ (A( s)^* \alpha)_{(u,t)} ( v ,  T) = 
\tfrac12\omega(A_0 (s )u , A_0 ( s )v ) - H ( A_0 ( s ) u ) s^{-2} T =
s^{\frac{5-p}{p-1}} ( \tfrac12\omega (u , v ) -  H ( u ) T ) \,.\]
That means that the form is preserved for $ p = 5 $. For $ p \neq 5 $
we still preserve the kernel of $ \wo = d \alpha $ which is consistent
with \eqref{eq:sc} preserving the solutions. 

To see the invariant quantity given by Noether's theorem (formula 
\eqref{eq:Noeth}) for $ p = 5 $ we compute
\[ \frac{d}{ds} A ( s ) ( u , t ) |_{s=1} =  
( u/2 + xu_x , - 2t ) \,, \]
and the conserved quantity is 
\begin{equation}
\label{eq:cons}
 F ( u , t ) = \tfrac12\omega ( u , u/2 + x u_x ) + 2 t H( u ) = 
- \tfrac12\Im \int x u_x \bar u + 2 t H ( u ) \,, 
\end{equation}
which is a version of the virial identity, typically written
\begin{equation}
\label{E:virial1}
\partial_t \left( \Im \int xu_x \bar u \right) = 4H(u) \,.
\end{equation}

\subsection{Case of $ p=5 $}
\label{p5}

Here the scaling symmetry is part of a more general
scaling property:
\[ u ( t , x ) \longmapsto ( ct + d )^{-1/2} e^{\frac{ i c x^2}{ 2 ( ct + d)}}
u \left( \frac{ at+b} { ct+d} , \frac x{ct+d} \right) \,, \ \ 
\left( \begin{array}{ll} a & b \\ c & d \end{array} \right) \in 
SL_2 ( \RR ) \,, \]
see \cite{KS} for this and a recent study of the quintic NLS.

Motivated by this,
for $ g \in SL_2 ( \RR ) $ we define the standard action on $ \overline \RR $:
\[  g ( t ) = \frac{ a t + b }{ ct + d } \,, \ \ 
\left( \begin{array}{ll} a & b \\ c & d \end{array} \right) \in 
SL_2 ( \RR )  \,.\]
Then 
\[ A ( g ) \; : \; \wV \rightarrow \wV \,, \]
is given as follows
\[ A ( g ) ( u , t ) = ( A_0 ( g ) u  ,  g^{-1}(t ) ) \,, \ \ 
A_0 ( g ) u = (g'(t))^{-\frac14} e^{ -i g''(t) x^2/(4 g'(t)) } 
u \left( ( g'(t))^{\frac 12}  x  \right) \,. \]
Since $ g'( t) = ( ct + d )^{-2} $, $ (g'(t))^{\frac 12} = 
( c t + d )^{-1} $ is well defined.

The cases of
\[ \left( \begin{array}{ll} a &  \; 0\\ 0 & 1/a \end{array} \right)  \,, \ \
\left( \begin{array}{ll} 1 & b \\ 0 & 1 \end{array} \right) \,, \]
correspond to scaling and translation with the invariant quantities already 
discussed. 

For 
\[ g ( s ) \defeq \left( \begin{array}{ll} \cos s 
 & - \sin s  \\ \sin s & \ \cos s \end{array} \right) \,, \]
we obtain
\[ \frac{d}{ds} ( A ( g ( s )) ( u, t ) |_{s=0} = 
( (-t/2 + i x^2/2 ) u - t x u_x , 1 + t^2 ) \,, \]
so that the conserved quantity is 
\[ F( u , t) = - \frac14 \int x^2 |u|^2 dx + \frac12t \Im \int x u_x \bar u 
- H(u)t^2-H(u) \,.  \]
Since $ H( u ) $ is conserved and we also have \eqref{eq:cons}
we conclude that
\begin{equation}
\label{E:virial2}
\int x^2 |u ( x , t ) |^2 dx = \int x^2 |u ( x , 0 ) |^2 dx 
+ 2t \Im \int x u_x \bar u - 4H(u)t^2  \, 
\end{equation}
which is again a version of the virial identity.  This version of the virial identity is usually written
$$\partial_t^2 \int x^2 |u(x,t)|^2 \, dx = 8H(u)$$
Two time integrations, then substituting the identity
$$\partial_t \int x^2 |u|^2 \, dx = 2\Im \int x \bar u \partial_x u \, dx$$ 
evaluated at $t=0$, give
$$\int x^2|u(t)|^2 \,dx = \int x^2|u|^2 \,dx \Big|_{t=0}+ 2\Im \int x \bar u \partial_x u \, dx \Big|_{t=0} t + 4H(u)t^2$$
Integrating \eqref{E:virial1} from $0$ to $t$ gives an expression for $\Im \int x \bar u \partial_x u \, dx \Big|_{t=0}$, which substituted here gives \eqref{E:virial2}.

\section{The least action principle}
\label{lap}

To formulate the least action principle we need to define the 
Lagrangian. In the last section, although we took the 
Lagrangian point of view, we used the form $ \alpha $ given by 
\eqref{eq:defal}. The Lagrangian,
\[  {\mathcal L} \; : \; T \wV \; \longrightarrow \; \RR \,, \]
is defined as follows:
\[ {\mathcal L} ( u , t , X , T ) \defeq \alpha_{(u,t)} ( X , T ) \,, \ \
X \in T_u V \,. \]
If $ t \mapsto u $ is a curve in $ V$ we use a simplified notation

\begin{equation}
\label{eq:defla0} 
 {\mathcal L} ( u ) \defeq \alpha_{(u,t)} ( \dot u ,  1) \,.
\end{equation}

For the equation \eqref{eq:NLS} we obtain 
\begin{equation}
\label{eq:defla} 
 {\mathcal L} ( u ) = \frac12 \omega ( u, \dot u ) - H( u ) =
- \frac 12 \Im \int u_t \bar u - \frac 14 \int |u_x|^2 + \frac1{p+1} 
\int |u|^{p+1} \,. 
\end{equation}

Action is more natural than considering Lagrangian. Let $ \gamma $
be a curve in $ \wV $. Then the action on $ \gamma $ is defined as
\begin{equation}
\label{eq:defac} {\mathcal S}_\gamma \defeq \int_\gamma \alpha \,. 
\end{equation}
When the curve is given by $ t \mapsto ( u ( t) , t ) $ we 
get, in the notation of \eqref{eq:defla0},
\[  {\mathcal S }_\gamma  \defeq \int {\mathcal L} ( u ) dt \,. \]

The least action principle can be formulated as follows:

\begin{prop}
The curve $ \gamma: s \mapsto ( u ( s ) , t ( s) ) $ is 
critical for $ {\mathcal S}_\gamma $ if and only if 
$ \dot \gamma (s) \in \ker \wo_{\gamma (s)} $. In other words
\begin{equation}
\label{eq:delta} \delta {\mathcal S}_\gamma = 0 \; 
\Longleftrightarrow \;  \dot \gamma (s) \in \ker \wo_{\gamma (s)} \,. 
\end{equation}
\end{prop}
\begin{proof}
We first give the proof in finite dimensions.
Let $ \gamma_r $ be a smooth family of curves such that $ \gamma_0 = 0 $, 
and $ \gamma_r $ is equal to $ \gamma $ outside of a compact subset, 
disjoint from $ \partial \gamma $. 
Being stationary means that for any such family,
\[ \frac { d} { d r} \int_{\gamma_r } \alpha \; \; \Big|_{r=0} = 0 \,.\]
Let $ F_r $ be a smooth family of diffeomorphism such that, for $ r $
small, $ \gamma_r = F_r ( \gamma ) $, and let $ X = (d/dr) F_r |_{r=0} $
be a vector field defined on $ \gamma $. Then, as in the proof of
Proposition \ref{p:Noeth}, we use Cartan's formula:
\[ \begin{split} 
 \frac { d} { d r} \int_{\gamma_r } \alpha \; \Big|_{r=0} 
& = \frac { d} { d r} \int_{\gamma } F_r^* \alpha \; \Big|_{r=0} 
 = \int_\gamma ( d \alpha ( X , \bullet ) + d ( \alpha ( X ) ) ) \\
& = \int_\gamma \widetilde \omega ( X , \bullet )  + \alpha( X )|_{ \partial 
\gamma } 
 = \int_\gamma \widetilde \omega ( X , \bullet )  
\,,
\end{split}\]
since by the assumptions on $ \gamma_r $, $ X \equiv 0 $ near 
$ \partial \gamma$. This means that 
\[  \frac { d} { d r} \int_{\gamma_r } \alpha \; \Big|_{r=0} = 0 
\ \Longrightarrow \ \widetilde \omega_{\gamma(s)} ( X_{\gamma(s)}  , \dot \gamma(s) ) = 0 \ \forall \; X \ \forall \; s \,, \]
which proves the proposition in finite dimensions.

The same formal argument applies to evolution equation and in our 
case we check it by a sandard direct computation:
\[ \begin{split} 
{\mathcal S } ( u + \delta u , t + \delta t ) & =
\int ( \omega ( u + \delta u , \dot u + \delta \dot u ) /2 - 
H ( u + \delta u ) ( \dot t + \delta \dot t ) ) ds 
\,.
\end{split} \]
Integrating by parts and neglecting higher order terms 
we obtain the first variation of
$ {\mathcal S } $:
\[ \begin{split} \delta {\mathcal S} 
&  = \int ( \omega ( \delta u , \dot u ) - \dot t d_u H  ( \delta u ) +
d_u H ( \dot u ) \delta t ) ds \\
& = \int \wo_{(u,t)} ( ( \delta u , \delta t ), ( \dot u , \dot t ) ) ds \,, 
\end{split} \]
and this vanishes for all $ \delta u $ and $ \delta t $ if and only 
if 
$ ( \dot u , \dot t ) \in \ker \wo_{(u,t)} $.
\end{proof}

\section{Effective dynamics}
\label{ed}

Suppose that $ \wM \subset \wV $ is a submanifold which 
is {\em presymplectic} in the sense that
\begin{equation}
\label{eq:const}   \dim \ker \wo \rest_{ \wM } = k \,, \ \ \text{where $ k $ is 
 constant on $ \wM $.} \end{equation}
Then $ \ker \wo \rest_{\wM } $ defines a foliation of $ \wM $ with 
leaves of dimension $ k $. We note that the fact that $ d \wo = 0 $
and the formula for $ d \rho ( X, Y , Z ) $,
\[  X \rho ( Y , Z ) - Y \rho ( X ,  Z) + 
Z \rho ( X , Y ) - \rho ( [ Y , Z ] , X ) + \rho ( [ X , Z ], Y ) 
- \rho ( [ X, Y ] , Z ) \,, \]
show that the $ \ker \wo $ satisfies the Frobenius integrability 
condition.

The method of collective coordinates for motion close to $ \widetilde M $
is based on the following principle:
\begin{quote}
Suppose that $ \gamma $ is critical for $ {\mathcal S} $ 
(for instance  $ t \mapsto u ( t ) $ which satisfies 
\eqref{eq:HL} or, 
equivalently, \eqref{eq:Hflow}). Suppose also that $ \gamma $
is {\em close} to $ \wM $. Then it is close to a fixed leaf of the
above foliation. 
\end{quote}

Here is a trivial example to illustrate this. Let $ V = T^* \RR $
and $ H ( x , \xi ) = \xi^2/2 $. Then suppose that $ \wM = \{ \xi = 0 \}$.
In that case $ \dim \ker \wo \rest_\wM $ is $ 3 $. If 
\[  \gamma ( t ) = (  ( x + t\epsilon  , \epsilon ), t ) \,, \]
then it is close $ \wM $, which the only leaf of the foliation. 

What one normally wants is (see \cite{OS} for examples from the 
physics literature and \cite{FrY} for an implicit application of this
principle in the mathematics literature):
\begin{quote}
Let $ \wM $ satisfy \eqref{eq:const} with $ k = 1 $. 
Suppose that $ \gamma $ is critical for $ {\mathcal S} $.
Suppose also that $ \gamma $
is {\em close} to $ \wM $. Then $ \gamma $ is close to a 
$ \gamma_{\wM} \subset \wM $ which is critical for 
$ {\mathcal S}_{\gamma_0 } $, $ \gamma_0 \subset \wM $.
In other words, we restrict the Lagrangian to the submanifold
and compute the action there. 
\end{quote}

The simplest case is given by $ \wM = M \times \RR $ with 
$ M $ symplectic, that is $ M$ for which $ \omega \rest_M $ is nondegenerate.
In that case the foliation is given by 
\[   s \rightarrow ( \exp ( s \Xi_{H\;\rest_{M}} ), s ) \,, \]
where $ \Xi_{p }$ is the Hamilton vector field of a Hamiltonian $ p $.
This is very clear in finite dimensions since then, locally, 
\[ M = { ( x , \xi ) : x''=\xi''=0 }  \,, \ \ x = ( x' , x'') \,, \ \
\xi = ( \xi' , \xi'') \,, \]
and 
\[ \wo \rest_{\wM } = ( d\xi' + H_{x'} ( x', 0 , \xi' , 0 ) dt ) \wedge
 ( dx' - H_{\xi'}  ( x', 0 , \xi' , 0 )  dt ) \,.\]

However we may have situations in which $ \omega \rest_M $ is degenerate
yet $ \ker \wo \rest_{M \times \RR } $ keeps fixed rank $ 1$. That means that the 
Hamiltonian formalism is not applicable but the Lagrangian one is.
Here is a simple example:
\begin{gather*} 
M = \{ ( x_1 , 0  , \xi_2^2 , \xi_2  ) \} \subset V = T^* \RR^2\,, 
\ \  H ( x, \xi) = x_1  \,, \\
 \omega \rest_M =  2 \xi_2 d \xi_2 \wedge dx_1  \,, \ \
\dim \ker  \omega \rest_{ M \cap \{\xi_2 = 0 \} } = 2 = \dim M \,, \\ 
\wo \rest_\wM = ( 2 \xi_2  d \xi_2 + dt )\wedge dx_1  \ \ 
\dim \ker \wo \rest_\wM = 1 \,. 
\end{gather*}

A complicated example from the physics literature comes from \cite[\S 5]{GHW}.
If one considers $ M \subset V = H^1 ( \RR , \CC ) $ given by 
all sums of time-modulated solitons and time modulated nonlinear ground states
(see \cite[(5.1)]{GHW}),
\begin{gather}
\label{eq:ghw}
\begin{gathered}
 u(x) = u_S (x; \eta, Z, V, \phi) + u_D (x; a, \phi, \psi) \,, \\
u_S ( x  ; \eta, Z, V , \phi) \defeq
\eta \,  \sech ( \eta x - Z) e^{ i Vx - i \phi} \,, \\
u_D (x; a, \phi, \psi) \defeq
 a \, \sech \left( a x + \tanh^{-1} \left(\frac \gamma a\right)\right) 
e^{ - i ( \phi + \psi ) } \,,
\end{gathered}
\end{gather}
then in the reduced six dimensional space described by $ ( \eta, Z, V , \phi, 
a, \psi ) $ is {\em not} symplectic with respect to $ \omega $ given by 
\eqref{eq:omega}. This can be checked by computing the determinant of
a matrix corresponding to $ \omega \rest_M $ -- see Fig.\ref{f:2}.
Despite that the method of collective coordinates is used by the 
authors in constructing an effective Lagrangian \cite[(5.4)]{GHW}
and it is then used to obtained approximate equations of motion 
\cite[(5.9)]{GHW}. It is numerically 
shown to give a good agreement with the solution of the equation.

\begin{figure}
\begin{center}
\includegraphics[width=5.5in]{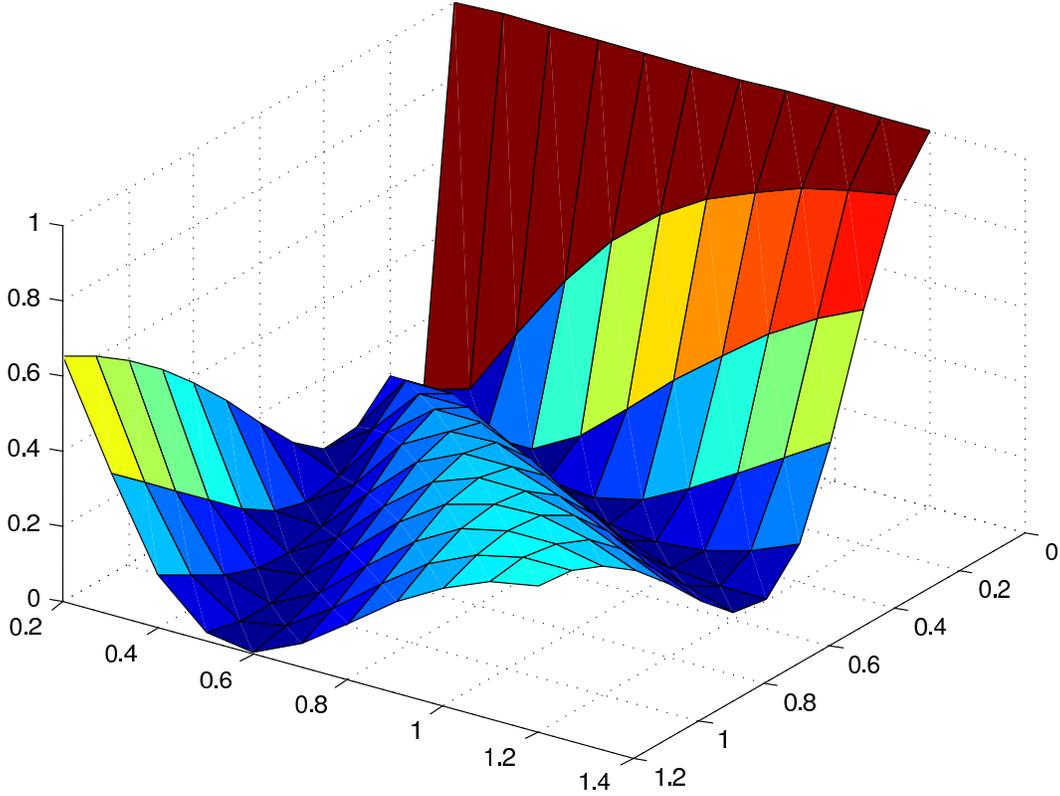}
\end{center}
\caption{The plot of $ \det ( \omega|_{M \cap \{ V = Z = 0 \}}  ) $ for $ \gamma = 0.1 $ (in the notation of \cite{GHW}), $ \phi = 0 $, $ \psi = \pi/4 $, 
$ 0 < \eta < 1.2 $ and $ 0.11 < a < 1.2$. The two lines along
which $ \omega|_{M \cap \{ V = Z = 0 \} } $ is degenerate are clearly visible.
The restriction to $ Z = V = 0 $ is not essential since these two variables
are essentially conjugate.}
\label{f:2}
\end{figure}

Finally we comment on the effective dynamics of solitons interacting
with slowly varying potentials. In \cite{HZ1} we followed \cite{FrSi}
and used a symplectic approach improving the results of \cite{FrSi}
and \cite{FrY} (same method apply to in that setting)
by obtaining equations of motion without errors and 
obtaining a better accuracy of approximation by a moving soliton ($ h \rightarrow h ^2 $, where $ h $ is slowness parameter of the potential). 
When attempting to reproduce these results using the Lagrangian
formalism we could obtain the same equations of motions (we later learned
that they were implicit in \cite{OS}), but could not obtain the 
$ h \rightarrow h^2 $ improvement. 

\medskip

\noindent 
{\sc Acknowledgments.}
We would like to thank Alan Weinstein for suggesting Souriau's monograph
as a source of geometric interpretation of Noether's theorem and Lagrangian
mechanics.
The work of the first author was supported in part by an NSF postdoctoral
fellowship, and that
of the second second author by the NSF grant DMS-0654436.

\end{document}